\documentclass[11pt]{amsart}
\usepackage{a4wide}
\usepackage[utf8]{inputenc}
\usepackage{amsmath}
\usepackage{amssymb}
\usepackage{amsfonts}
\usepackage{amsopn}
\usepackage{graphicx}
\usepackage{enumerate}
\usepackage{color}
\usepackage{mathtools}
\usepackage{mathrsfs}
\usepackage{bbm}
\usepackage{yhmath}
\usepackage[colorlinks,linkcolor=blue,anchorcolor=blue,citecolor=blue]{hyperref}

\newtheorem{theorem}{Theorem}[section]
\newtheorem{proposition}[theorem]{Proposition}
\newtheorem{lemma}[theorem]{Lemma}

\theoremstyle{definition}
\newtheorem{example}[theorem]{Example}
\theoremstyle{remark}
\newtheorem{remark}[theorem]{Remark}

\numberwithin{equation}{section}

\newcommand{\set}[1]{\left\{ #1 \right\}}

\newcommand{\C}{\mathbb{C}}
\newcommand{\R}{\mathbb{R}}
\newcommand{\Q}{\mathbb{Q}}
\newcommand{\Z}{\mathbb{Z}}
\newcommand{\N}{\mathbb{N}}

\newcommand{\f}{\infty}

\newcommand{\ep}{\varepsilon}

\newcommand{\p}{\mathfrak{p}}
\newcommand{\calO}{\mathcal{O}}
\newcommand{\ord}{\mathrm{Ord}}
\newcommand{\n}{\boldsymbol{n}}
\newcommand{\ii}{\mathrm{i}}

\title{Rational points in Cantor sets in the complex plane}

\author{Wenxia Li}
\address[Wenxia Li]{School of Mathematical Sciences, Key Laboratory of MEA (Ministry of Education) $\&$ Shanghai Key Laboratory of PMMP, East China Normal University, Shanghai 200241, People's Republic of China}
\email{wxli@math.ecnu.edu.cn}

\author{Zhiqiang Wang}
\address[Zhiqiang Wang]{College of Mathematics and Statistics, Center of Mathematics, Chongqing University, Chongqing 401331, People's Republic of China \& Department of Mathematics, University of British Columbia, Vancouver, British Columbia, V6T 1Z2, Canada}
\email{zhiqiangwzy@163.com,~zqwangmath@cqu.edu.cn}

\author{Jiuzhou Zhao}
\address[Jiuzhou Zhao]{School of Mathematics and Statistics, Key Laboratory of Engineering Modeling and Statistical Computation of Hainan Province, Hainan University, Haikou 570228, People’s Republic of China}
\email{zhao9zone@gmail.com}

\date{\today}
\subjclass[2020]{11A63, 28A80}

\begin{document}

\begin{abstract}
Let $K$ be an imaginary quadratic field and let $\mathcal{O}_K$ be the ring of algebraic integers of $K$.
For $\alpha \in \mathcal{O}_K$ with $|\alpha| > 1$, define \[ \mathcal{D}_\alpha = \bigcup_{n=0}^\f \frac{\mathcal{O}_K}{\alpha^n}. \]
For $\beta \in \mathcal{O}_K$ with $|\beta|>1$ and a finite subset $A \subset \mathcal{O}_K$, define \[ S_{\beta,A} = \bigg\{ \sum_{k=1}^{\f} \frac{a_k}{\beta^k}: \; a_k \in A \;\forall k \in \mathbb{N} \bigg\}. \]
Suppose that $\alpha$ and $\beta$ are relatively prime.
In this paper, we show that if $\dim_{\mathrm{H}} S_{\beta,A} < 1$, then the intersection $\mathcal{D}_\alpha \cap S_{\beta,A}$ is a finite set. In general, the threshold for the Hausdorff dimension of $S_{\beta,A}$ is sharp.
If we further assume that $\mathcal{O}_K$ is a unique factorization domain and that $\overline{\alpha}$ and $\alpha$ are relatively prime, then we establish the finiteness of the intersection under the weaker condition $\dim_{\mathrm{H}} S_{\beta,A} < 2$.
This extends the previously known results on the real line.
\end{abstract}
\keywords{radix expansion, Cantor set, imaginary quadratic field}

\maketitle

\section{Introduction}

\subsection{Rational points in Cantor sets}

Let $\N$ be the set of all positive integers, and let $\N_0:=\N\cup\set{0}$.  For $k\in\N$, let $\N_{\ge k}:=\N\cap[k,+\f)$.
Let $\#A$ denote the cardinality of a set $A$.
The greatest common divisor of $m,n \in \N$ is denoted by $\gcd(m,n)$.


Given $q\in \N_{\ge 3}$ and a finite set $A \subset \mathbb{Q}$ with $\#A \ge 2$,
we consider the \emph{iterated function system} (IFS) ${\mathcal F}_{q,A}=\big\{ \phi_a(x)=(x+a)/q: a\in A \big\}$. According to Hutchinson \cite{Hutchinson-1981}, there exists a unique non-empty compact subset $S_{q,A}$ of $\R$ such that \[ S_{q,A} = \bigcup_{a \in A} \phi_a \big( S_{q,A} \big), \]
which is called a \emph{self-similar set}. The set $S_{q,A}$ can be written algebraically as
\begin{equation*}
  S_{q,A}=\bigg\{\sum_{k=1}^{\infty} \frac{a_k}{q^{k}}:\;a_k \in A \;\;\forall k\in \N\bigg\}.
\end{equation*}
For $p\in \N_{\ge 2}$, let $D_p$ be the set of all rational numbers in $\mathbb {R}$ having a finite $p$-ary expansion. That is,
\begin{equation*}
D_p=\bigcup_{n=1}^\f \frac{\mathbb {Z}}{p^n},
\end{equation*}
which is a proper subset of $\mathbb Q$ and  is dense in $\mathbb R$.

The finiteness of the intersection of $D_p$ and $S_{q,A}$ under the assumption that $\gcd(p,q) = 1$, has been extensively studied \cite{Wall-1990,Nagy-2001,Bloshchitsyn-2015,Schleischitz-2021,Shparlinski-2021,Li-Li-Wu-2023,
JKLW-2024,Kong-Li-Wang-2025}.
In the solution to Problem H-339 in \cite{Whitney-Wall-1983}, Wall first showed that $$D_{2} \cap S_{3,\{0,2\}} = \bigg\{0, \frac{1}{4}, \frac{3}{4},1 \bigg\}.$$
And Wall \cite{Wall-1990} also proved that
$$D_{10} \cap S_{3,\{0,2\}} = \bigg\{ 0,\frac{1}{4}, \frac{3}{4}, \frac{1}{10}, \frac{3}{10}, \frac{7}{10}, \frac{9}{10}, \frac{1}{40}, \frac{3}{40}, \frac{9}{40}, \frac{13}{40}, \frac{27}{40}, \frac{31}{40}, \frac{37}{40}, \frac{39}{40},1  \bigg\}.$$
Note that $S_{3,\{0,2\}}$ is the classical middle third Cantor set.
Later, Nagy \cite{Nagy-2001} showed that the intersection $D_p \cap S_{3,\{0,2\}}$ is finite for each prime $p > 3$.
Bloshchitsyn \cite{Bloshchitsyn-2015} generalized this result and proved that if $p > q^2$ is a prime and $A \subset \{0,1,\ldots, q-1\}$ with $\#A< q$, then $D_p \cap S_{q,A}$ is a finite set.
The general result was proved by Schleischitz \cite[Corollary 4.4]{Schleischitz-2021}: if $\gcd(p,q)=1$, and $A \subset \{0,1,\ldots, q-1\}$ with $\#A< q$, then the intersection $D_p \cap S_{q,A}$ is finite, a stronger version of which was obtained by Shparlinski \cite{Shparlinski-2021} using a different method.
Recently, Kong, Li and the second author \cite{Kong-Li-Wang-2025} established the same conclusion for $A \subset \Q$.

\begin{theorem}[\cite{Schleischitz-2021,Shparlinski-2021,Kong-Li-Wang-2025}]
\label{theorem-real-line}
  Let $p\in \N_{\ge 2}$, $q \in \N_{\ge 3}$ and let $A \subset \Q$ be a finite set.
  Suppose that $\gcd(p,q)=1$.
  Then the intersection $D_p \cap S_{q,A}$ is finite if and only if $\dim_{\mathrm{H}} S_{q,A} < 1$.
\end{theorem}

Here, $\dim_{\mathrm{H}} E$ denotes the Hausdorff dimension of a set $E$, and we refer the reader to \cite{Falconer2014} for the definition and properties of Hausdorff dimension.
It is worth pointing out that Theorem \ref{theorem-real-line} has been applied to study the spectral eigenvalue problems for self-similar spectral measures \cite{Kong-Li-Wang-2025,Wang-2025}.

The main purpose of this paper is to extend the sufficiency part of Theorem \ref{theorem-real-line} to the complex plane.

\subsection{Algebraic number theory}

Here, we recall some basic notions in algebraic number theory, and we refer the reader to \cite{book-ANT-2022} for further details.

Let $K$ be an algebraic number field, and let $\mathcal{O}_K$ denote the ring of algebraic integers of $K$.
Note that $\mathcal{O}_K$ is a \emph{Dedekind domain}, which means that every non-zero proper ideal of $\mathcal{O}_K$ can be written uniquely as a product of prime ideals of $\mathcal{O}_K$ except for the order of factors (cf. \cite[Section 3.2]{book-ANT-2022}).

Let $R$ be an integral domain and let $I$ be a non-zero ideal of $R$.
For $a, b \in R$, we say that $a$ is \emph{congruent to $b$ module $I$}, denoted by $a \equiv b \pmod{I}$, if $a-b \in I$.
Let $(R/I)^\times$ denote the unit group of the quotient ring $R/I$, which consists of all invertible elements for the multiplication in $R/I$.
For two ideals $I_1,I_2$ of $R$, we say that \emph{$I_1$ divides $I_2$}, denoted by $I_1 \mid I_2$, if there exists an ideal $I_3$ of $R$ such that $I_2 = I_1 I_3$.
For two ideals $I_1,I_2$ of $\mathcal{O}_K$, we have $I_1 \mid I_2$ if and only if $I_2 \subset I_1$ (cf. \cite[Theorem 3.19 (i)]{book-ANT-2022}).
Two ideals $I_1,I_2$ of $\mathcal{O}_K$ are said to be \emph{relatively prime} if $I_1 + I_2 = \mathcal{O}_K$.
Two elements $\alpha,\beta \in \calO_K$ are said to be relatively prime if $\alpha \calO_K$ and $\beta\calO_K$ are relatively prime, i.e., $(\alpha \calO_K)+(\beta\calO_K) = \calO_K$.

Let $I$ be a non-zero ideal of $\mathcal{O}_K$. The quotient ring $\mathcal{O}_K / I$ is always finite (cf. \cite[Example 3.25]{book-ANT-2022}), and the number of elements of $\calO_K/I$ is called \emph{the norm of $I$}, denoted by $N(I)$.
So $(\mathcal{O}_K / I)^\times$ is a finite group.
For $\beta \in \calO_K$, the notation $\beta \in (\mathcal{O}_K / I)^\times$ means that the equation $\beta x \equiv 1 \pmod{I}$ is solvable in $\mathcal{O}_K$. Then \emph{the order of $\beta$ modulo $I$}, denoted by $\mathrm{Ord}_I(\beta)$, is defined to be the smallest integer $n \in \N$ such that $\beta^n \equiv 1 \pmod{I}$.

Let $\mathfrak{p}$ be a non-zero prime ideal of $\mathcal{O}_K$.
Then $\mathcal{O}_K / \mathfrak{p}$ is a finite field.
So $\mathfrak{p}$ contains a unique prime $p \in \N$ which is the characteristic of the finite field $\mathcal{O}_K / \mathfrak{p}$.
It follows that $\mathfrak{p}$ contains $p \mathcal{O}_K$ and hence $\mathfrak{p} \mid (p \mathcal{O}_K)$.
The largest integer $e \in \N$ such that $\mathfrak{p}^e \mid (p \mathcal{O}_K)$ is called the \emph{index of ramification of $\mathfrak{p}$}.
Note that $N(\mathfrak{p}) = \# \big( \mathcal{O}_K / \mathfrak{p} \big) = p^f$ for some $f \in \N$. The positive integer $f$ is called the \emph{residual degree of $\mathfrak{p}$}.

\subsection{Main results}

Let $K = \Q(\sqrt{d})$ be an imaginary quadratic field, where $d$ is a negative square free integer.
By \cite[Theorem 2.9]{book-ANT-2022}, we have
\begin{equation}\label{eq:O-K}
\mathcal{O}_K =
\begin{cases}
	\Z[\sqrt{d}], & \text{if}\;\; d \equiv 2 \;\text{or}\; 3 \pmod{4}; \\
	\Z\big[ \frac{1+\sqrt{d}}{2} \big], & \text{if}\;\; d \equiv 1 \pmod{4}.
\end{cases}
\end{equation}
It is easy to check that for each $z \in \calO_K$, we have $|z|^2 \in \Z$.

For $\beta \in \mathcal{O}_K$ with $|\beta|>1$ and a finite subset $A \subset \mathcal{O}_K$  with $\#A\ge 2$, let $S_{\beta,A}$ be the self-similar set in the complex plane generated by the IFS \[ \mathcal{F}_{\beta,A} = \Big\{ \varphi_a(z) = \frac{z+a}{\beta}: a \in A  \Big\}. \]
Similarly, the set $S_{\beta,A}$ can be written as
\begin{equation}\label{eq:S-beta}
  S_{\beta,A} = \bigg\{ \sum_{k=1}^{\f} \frac{a_k}{\beta^k}: \; a_k \in A \;\forall k \in \N \bigg\}.
\end{equation}
For $\alpha \in \mathcal{O}_K$ with $|\alpha| > 1$, define \[ \mathcal{D}_\alpha = \bigcup_{n=1}^\f \frac{\mathcal{O}_K}{\alpha^n}. \]
Let $\overline{z}$ denote the complex conjugate of $z \in \C$.
Our main theorem is the following.

\begin{theorem}\label{theorem-complex-plane}
  Let $K$ be an imaginary quadratic field.
  Let $\alpha, \beta \in \calO_K$ with $|\alpha| > 1$ and $|\beta|>1$.
  Let $A \subset \calO_K$ be a finite set with $\# A \ge 2$.
  Suppose that $\alpha$ and $\beta$ are relatively prime.

  {\rm(i)} If $\dim_{\mathrm{H}} S_{\beta,A} < 1$, then we have $ \#\big( \mathcal{D}_\alpha \cap S_{\beta,A} \big) < + \f$.

  {\rm(ii)} Suppose further that $\calO_K$ is a unique factorization domain, and that $\overline{\alpha}$ and $\alpha$ are relatively prime. If $\dim_{\mathrm{H}} S_{\beta,A} < 2$, then we have $\#\big( \mathcal{D}_\alpha \cap S_{\beta,A} \big) < + \f$.
\end{theorem}
\begin{remark}
  (a) Note that $K$ is the quotient field of $\calO_K$ (cf. \cite[Theorem 1.6 (iii)]{book-ANT-2022}).
  For a finite subset $A \subset K$, we can find $0 \ne \theta \in \calO_K$ such that $\theta A \subset \calO_K$.
  Note that \[ \mathcal{D}_\alpha \cap S_{\beta,A} = \frac{1}{\theta} \Big( \big( \theta \mathcal{D}_\alpha \big) \cap S_{\beta,\theta A} \Big) \subset \frac{1}{\theta} \big(  \mathcal{D}_\alpha \cap S_{\beta,\theta A} \big). \]
  Thus, Theorem \ref{theorem-complex-plane} also holds for any finite subset $A \subset K$.

  (b) For $\alpha, \beta \in \calO_K$ with $|\alpha|,|\beta| > 1$, if $\gcd(|\alpha|^2,|\beta|^2)=1$, then we have $(\alpha\overline{\alpha}\calO_K) + (\beta\overline{\beta} \calO_K) = \calO_K$.
  It follows that $(\alpha\calO_K) + (\beta\calO_K) = \calO_K$.
  Thus, we obtain that $\alpha$ and $\beta$ are relatively prime in $\calO_K$.

  (c) For an imaginary quadratic field $K=\Q(\sqrt{d})$, where $d$ is a negative square free integer, $\mathcal{O}_K$ is a unique factorization domain only for $d=-1,-2,-3,-7,-11,-19,-43,-67,-163$ (cf. \cite[Remark 3.6]{book-ANT-2022}).

  (d) Let $K = \Q(\ii)$, where $\ii = \sqrt{-1}$ denotes the imaginary unit. Then $\calO_K = \Z[\ii]=\big\{n+ m \ii: n,m \in \Z \}$ is the ring of Gaussian integers.
  Let $\alpha = p \in \N_{\ge 2}$, $\beta=q \in \N_{\ge 3}$, and $A \subset \Q$ a finite set.
  Note that $S_{q,A} \subset \R$. So we have $\mathcal{D}_p \cap S_{q,A} = D_p \cap S_{q,A}$.
  If $p$ and $q$ are coprime in $\N$, then $p$ and $q$ are also relatively prime in $\calO_K$.
  If, moreover, $\dim_{\mathrm{H}} S_{q,A} < 1$, then by Theorem \ref{theorem-complex-plane} (i) and the remark (a), we conclude that \[ \#\big( D_p \cap S_{q,A} \big) < + \f. \]
  Thus, the sufficiency part of Theorem \ref{theorem-real-line} is a corollary of Theorem \ref{theorem-complex-plane} (i).
  This also shows that the threshold for the Hausdorff dimension of $S_{\beta,A}$ in Theorem \ref{theorem-complex-plane} (i) is sharp.

  (e) The proof of Theorem \ref{theorem-complex-plane} follows the spirit of \cite{Schleischitz-2021,Kong-Li-Wang-2025}. Here, $\calO_K$ is not a unique factorization domain in general. We need to deal with the prime ideal factorization.

  (f) While preparing this manuscript, we became aware that Wu \cite{Wu-2025+} independently proved Theorem \ref{theorem-complex-plane} (i) in the setting $K=\Q(\ii)$ and $\calO_K = \Z[\ii]$.
\end{remark}

Finally, we give a concrete example in the Gaussian integer ring $\Z[\ii]$.
For $\theta= -n +\ii$ with $n \in \N$, we have a canonical number system $\big( \theta, \{0,1,\ldots, n^2\} \big)$ in $\Z[\ii]$ \cite{Katai-Szabo-1975}.
This means that every Gaussian integer $\gamma \in \Z[\ii]$ can be written uniquely as
\[ \gamma = \sum_{k=0}^{\ell} \xi_k \theta^k \;\;\text{with}\;\; \xi_0,\xi_1, \ldots, \xi_\ell \in \{0,1,\ldots, n^2\}.\]

\begin{example}
  Let $\alpha= -n + \ii$ with $n \in \N$.
  Then we have \[ \mathcal{D}_\alpha = \bigg\{ \sum_{k=-\ell}^{\ell} \xi_k \alpha^{k}: \;\ell \in \N,\; \xi_k \in \{0,1,\ldots,n^2\} \; \forall -\ell \le k \le \ell  \bigg\}. \]
  Suppose that $n$ is an even integer.
  This implies that $\alpha$ and $\overline{\alpha}$ are relatively prime in $\Z[\ii]$.
  Let $\beta = -m + \ii$ with $m \in \N$, and let $A$ be a proper subset of $\{ 0,1,\ldots, m^2 \}$. Then we have \[ \dim_{\mathrm{H}} S_{\beta,A} = \frac{\log \# A}{\log |\beta|} < 2.  \]
  Suppose that $\gcd(n^2+1,m^2 +1)=1$, which implies that $\alpha$ and $\beta$ are relatively prime in $\Z[\ii]$.
  For instance, we may take $\alpha = -4+\ii$, $\beta=-2+\ii$, and $A=\{0,1,2,3\}$.
  Then by Theorem \ref{theorem-complex-plane} (ii), the intersection $\mathcal{D}_\alpha \cap S_{\beta,A}$ is a finite set.
\end{example}

The rest of paper is organized as follows.
The proof of Theorem \ref{theorem-complex-plane} relies on two key propositions, which will be proved in Section \ref{sec:upper-bound} and \ref{sec:lower-bound}, respectively.
We finish the proof of Theorem \ref{theorem-complex-plane} in Section \ref{sec:proof}.

\section{Upper bound of period length for eventually periodic coding}\label{sec:upper-bound}

In this section, let $K$ be an imaginary quadratic field and let $\mathcal{O}_K$ be the ring of algebraic integers of $K$.
By (\ref{eq:O-K}), for any non-zero element $z \in \mathcal{O}_K$ we have $|z| \ge 1$.
This property plays an important role in the following proposition.

By (\ref{eq:S-beta}), each element $z\in S_{\beta,A}$ has at least one expression of the form
\begin{equation}\label{eq:coding}
  z= \sum_{k=1}^{\f} \frac{a_k}{\beta^k} \;\;\text{with each}\; a_k \in A.
\end{equation}
The infinite sequence $(a_k)_{k =1}^\f$ in (\ref{eq:coding}) is called a \emph{coding} of $z$.
The following proposition is inspired by \cite[Theorem 4.3]{Schleischitz-2021} and \cite[Lemma 2.3]{Kong-Li-Wang-2025}.

\begin{proposition}\label{lem:order-upper-bound}
Let $\beta \in \mathcal{O}_K$ with $|\beta|>1$ and let $A \subset \mathcal{O}_K$ be a finite subset with $\#A\ge 2$.
Write $s=\dim_{\mathrm{H}} S_{\beta,A}$.
Then for $\ep>0$ there exists a constant $c_1=c_1(\ep)>0$ such that any complex number $\xi = v/u \in S_{\beta,A}$ with $v,u \in \mathcal{O}_K$ admits an eventually periodic coding with the period length $\le c_1 |u|^{s+\ep}.$
\end{proposition}
\begin{proof}
Let $(a_k)_{k=1}^\f \in A^\N$ be a coding of $\xi$, i.e.,
\[\xi=\sum_{k=1}^{\f} \frac{a_k}{\beta^k}.\]
Consider the sequence $\{\xi_n\}_{n=0}^\f$ defined by
\begin{equation}\label{eq:sequence-xi}
\xi_0:=\xi,\quad \textrm{and} \quad\xi_n:=\varphi_{a_n}^{-1}(\xi_{n-1})= \beta \xi_{n-1} - a_n \;\;\text{for}\; n\in \N.
\end{equation}
Then we have \[ \xi_n = \sum_{k=1}^{\f} \frac{a_{n+k}}{\beta^k} \in S_{\beta,A} \quad \forall n \in \N_0. \]

Note that $\beta \in \mathcal{O}_K$ and $A \subset \mathcal{O}_K$.
By (\ref{eq:sequence-xi}) we can recursively show that $\xi_n$ has the form
$$\xi_n=\frac{v_n}{u} \quad\text{for some}\; v_n \in \mathcal{O}_K. $$
Let $\mathcal{C}:=\{\xi_n:n\in \N_0 \}$.
Note that for any non-zero element $z \in \mathcal{O}_K$ we have $|z| \ge 1$.
Thus for any two distinct elements $z_1, z_2 \in \mathcal{C}$, we have
\begin{equation}\label{eq:gap}
|z_1-z_2 |\ge \frac{1}{|u|}.
\end{equation}

For $\delta>0$ and a bounded set $F$, let $N_\delta(F)$ denote the smallest number of closed balls of radius $\delta$ that cover $F$.
The box-counting dimension of $F$ is defined by \[ \dim_{\mathrm{B}} E = \lim_{\delta \to 0^+} \frac{\log N_\delta(E)}{-\log \delta}\] if this limit exists (see \cite{Falconer2014} for more details).
Since $S_{\beta,A}$ is a self-similar set, by \cite[Corollary 3.3]{Falconer1997} we have \[ \dim_{\mathrm{B}} S_{\beta,A} = \dim_{\mathrm{H}} S_{\beta,A} =s. \]
For $\ep>0$, by the definition of box-counting dimension, there exists $\delta_0>0$ such that \[ N_\delta\big( S_{\beta,A} \big) \le \frac{1}{\delta^{s+\ep}}\quad \forall 0< \delta < \delta_0. \]
Thus, there is a constant $c=c(\ep)$ such that
\begin{equation}\label{eq:N-delta-bound}
  N_\delta\big( S_{\beta,A} \big) \le \frac{c}{\delta^{s+\ep}}\quad \forall 0< \delta < 1.
\end{equation}

Note that $\mathcal{C} \subset S_{\beta,A}$.
By (\ref{eq:gap}) we have \[ N_{\frac{1}{3|u|}} \big( S_{\beta,A} \big) \ge N_{\frac{1}{3|u|}} \big( \mathcal{C} \big) \ge \# \mathcal{C}. \]
This together with (\ref{eq:N-delta-bound}) implies that $\# \mathcal{C} \le c (3|u|)^{s+\ep}$.
That is, there exists a constant $c_1>0$ depending on $\ep$ such that
$$\#\mathcal{C} \le c_1|u|^{s+\ep}.$$
Thus, there are $0\le k < n \le c_1 |u|^{s+\ep}$ such that $\xi_{k} = \xi_{n}$. By (\ref{eq:sequence-xi}) it yields that
$$\xi_k =\varphi_{a_{k+1}}\circ \varphi_{a_{k+2}}\circ \cdots \circ \varphi_{a_n} (\xi_n)=\varphi_{a_{k+1}}\circ \varphi_{a_{k+2}}\circ \cdots \circ \varphi_{a_n} (\xi_k).$$
This means that $\xi_k$ has a periodic coding $(a_{k+1}a_{k+2}\ldots a_{n})^{\f}$.
It follows that $\xi$ has a eventually periodic coding $a_1a_2\ldots a_k (a_{k+1}a_{k+2}\ldots a_{n})^{\f}$ with the period length $n-k \le c_1 |u|^{s+\ep}$.
\end{proof}

\section{Lower bound of the order}\label{sec:lower-bound}

In this section, let $K$ be an algebraic number field and let $\mathcal{O}_K$ be the ring of algebraic integers of $K$.

Let $\mathfrak{p}$ be a non-zero prime ideal of $\mathcal{O}_K$ and $\beta \in \calO_K \setminus \p$.
For any $n \in \N$, suppose that $\beta \calO_K + \p^n \ne \calO_K$.
Then we can find a non-zero prime ideal $\p'$ of $\mathcal{O}_K$ such that $\beta \calO_K + \p^n \subset \p'$.
Note that $\p^n \subset \beta \calO_K + \p^n \subset \p'$.
This implies that $\p' = \p$. So we have $\beta \calO_K + \p^n \subset \p$.
It follows that $\beta \in \p$, a contradiction.
Thus, we conclude that $\beta \calO_K + \p^n = \calO_K$ for any $n \in \N$.
This means that the equation $\beta x \equiv 1 \pmod{\p^n}$ is solvable in $\calO_K$. That is, $\beta \in (\calO_K / \p^n)^\times$ and the order $\ord_{\p^n}(\beta)$ is well defined for any $n \in \N$.

Let $\p_1, \p_2, \ldots, \p_k$ be distinct non-zero prime ideals of $\mathcal{O}_K$ and $\beta \in \calO_K \setminus \bigcup_{j=1}^k \p_j$.
For any $n_j \in \N$, we have $\beta \in (\calO_K / \p_j^{n_j})^\times$.
Note that $\calO_K$ is a Dedeking domain.
By the Chinese remainder theorem (cf. \cite[page 60]{book-ANT-2022}), we have \[ \calO_K / \p_1^{n_1} \p_2^{n_2} \ldots \p_k^{n_k} \cong \calO_K / \p_1^{n_1} \oplus  \calO_K / \p_2^{n_2} \oplus \cdots \oplus \calO_K / \p_k^{n_k}. \]
Thus, we conclude that $\beta \in \big( \calO_K / \p_1^{n_1} \p_2^{n_2} \ldots \p_k^{n_k} \big)^\times$ and the order $\ord_{\p_1^{n_1} \p_2^{n_2} \ldots \p_k^{n_k}}( \beta )$ is well defined for any tuple $(n_1,n_2,\ldots,n_k) \in \N_0^k \setminus \{ \mathbf{0} \}$.

\begin{lemma}\label{lemma:order-prime}
  Let $\mathfrak{p}$ be a non-zero prime ideal of $\mathcal{O}_K$.
  Let $p \in \N$ be the unique prime contained in $\mathfrak{p}$, and let $e$ be the index of ramification of $\p$.
  For $\beta \in \calO_K \setminus \p$ with $|\beta| > 1$, there exists $n_0 \in \N$ such that \[ \ord_{\p^{n_0 + n}}(\beta) = p^{\lceil \frac{n}{e} \rceil} \cdot \ord_{\p^{n_0}}(\beta) \quad \forall n \in \N_0, \]
  where $\lceil x \rceil$ denotes the smallest integer not less than $x$.
\end{lemma}
\begin{proof}
  Write $m = \ord_{\p^{e+1}}(\beta)$. Then we have $(\beta^m -1) \in \p^{e+1}$, that is, $\p^{e+1} \mid (\beta^m -1)\calO_K$.
  Since $|\beta| > 1$, we have $\beta^m - 1 \ne 0$.
  Let $n_0$ be the largest positive integer $n$ such that $\p^n \mid (\beta^m -1)\calO_K$.
  Then we have $n_0 \ge e+1$ and $\ord_{\p^{n_0}}(\beta)=m$.

  We claim that
  \begin{equation}\label{eq:claim-1}
    \beta^{mp^k} - 1 \in \p^{n_0 + ke} \setminus \p^{n_0 + ke+1} \quad \forall k \in \N_0.
  \end{equation}
  We will prove the claim (\ref{eq:claim-1}) by induction on $k$.
  First, for $k=0$ it holds because of the maximality of $n_0$.
  Next, we assume that \[ \beta^{mp^k} - 1 \in \p^{n_0 + ke} \setminus \p^{n_0 + ke+1} \] for some $k \in \N_0$.
  So we can write \[ \beta^{mp^k} = 1 + \gamma \quad\text{with}\; \gamma \in \p^{n_0 + ke} \setminus \p^{n_0 + ke+1}. \]
  Then we have
  \begin{equation}\label{eq:k+1}
    \beta^{mp^{k+1}} = (\beta^{mp^k})^p = (1+\gamma)^p = 1 + p \gamma + \sum_{\ell=2}^{p} \binom{p}{\ell} \gamma^\ell.
  \end{equation}
  For $\ell\ge 2$, we have $\gamma^\ell \in \p^{(n_0 + ke)\ell} \subset \p^{2n_0 + ke} \subset \p^{n_0 + (k+1)e +1}$.
  It follows that
  \begin{equation}\label{eq:ell-2}
    \sum_{\ell=2}^{p} \binom{p}{\ell} \gamma^\ell \in \p^{n_0 + (k+1)e +1}.
  \end{equation}
  Note that $p \in \p^{e} \setminus \p^{e+1}$.
  Thus, we have $p\gamma \in \p^{n_0 + (k+1)e} \setminus \p^{n_0 + (k+1)e+1}$.
  Together with (\ref{eq:ell-2}) and (\ref{eq:k+1}), we conclude that \[ \beta^{mp^{k+1}} - 1 \in \p^{n_0 + (k+1)e} \setminus \p^{n_0 + (k+1)e+1}. \]
  That is, the claim (\ref{eq:claim-1}) holds for $k+1$.
  By induction on $k$, we prove the claim (\ref{eq:claim-1}).

  For any $k \in \N$, we first have $m = \ord_{\p^{n_0}}(\beta) \mid \ord_{\p^{n_0+ke}}(\beta)$.
  By (\ref{eq:claim-1}), we have \[ \beta^{mp^k} \equiv 1 \pmod{\p^{n_0+ke}} \;\;\text{and}\;\; \beta^{mp^{k-1}} \not\equiv 1 \pmod{\p^{n_0+ke}}. \]
  Note that $p$ is a prime.
  Thus, we conclude that
  \begin{equation}\label{eq:order-e}
    \ord_{\p^{n_0+ke}}(\beta) = m p^k \quad \forall k \in \N.
  \end{equation}

  For any $n \in \N$ with $e \nmid n$, write $n = k e + r$ where $k \in \N_0$ and $r \in \{1,2,\ldots, e-1\}$.
  Then we have $\ord_{\p^{n_0+ke}}(\beta) \mid \ord_{\p^{n_0+n}}(\beta)$ and $\ord_{\p^{n_0+n}}(\beta) \mid \ord_{\p^{n_0+(k+1)e}}(\beta)$.
  Note that $p$ is a prime.
  By (\ref{eq:order-e}) we obtain $\ord_{\p^{n_0+n}}(\beta) = mp^k \;\;\text{or} \;\; m p^{k+1}$.
  By (\ref{eq:claim-1}), we have $\beta^{mp^{k}} \not\equiv 1 \pmod{\p^{n_0+n}}$.
  Thus, we obtain that \[ \ord_{\p^{n_0+n}}(\beta) = m p^{k+1} = m p^{\lceil \frac{n}{e}\rceil}. \]
  Together with (\ref{eq:order-e}), we complete the proof.
\end{proof}


\begin{proposition}\label{prop:order}
  Let $\p_1, \p_2, \ldots, \p_k$ be distinct non-zero prime ideals of $\mathcal{O}_K$.
  For $1 \le j \le k$, let $p_j \in \N$ be the unique prime contained in $\mathfrak{p}_j$, and let $e_j$ be the index of ramification of $\p_j$.
  Let $\mathcal{P}=\big\{ p_j: 1 \le j \le k \big\}$ and for $p \in \mathcal{P}$, let $\mathcal{I}_p = \big\{ 1 \le j \le k: p_j = p \big\}$.
  Then for $\beta \in \calO_K \setminus \bigcup_{j=1}^k \p_j$ with $|\beta|>1$, there exists a constant $c_2 >0$ such that \[ \ord_{\p_1^{n_1} \p_2^{n_2} \ldots \p_k^{n_k}}( \beta )\ge c_2 \prod_{p \in \mathcal{P}} p^{\max\big\{ \big\lceil\frac{n_j}{e_j}\big\rceil: j \in \mathcal{I}_p \big\}}, \]
  for any tuple $(n_1,n_2,\ldots,n_k) \in \N_0^k \setminus \{ \mathbf{0} \}$.
\end{proposition}
\begin{proof}
  By Lemma \ref{lemma:order-prime}, for any $1 \le j \le k$, there exists $m_j \in \N$ such that
  \begin{equation}\label{eq:p-j}
    p_j^{\lceil \frac{n}{e_j} \rceil} \;\mid\; p_j^{m_j} \cdot \ord_{\p_j^{n}}(\beta) \quad \forall n \in \N.
  \end{equation}

  Fix a tuple $(n_1,n_2,\ldots,n_k) \in \N_0^k \setminus \{ \mathbf{0} \}$.
  Note that for $1 \le j \le k$, \[ \ord_{\p_j^{n_j}}(\beta) \; \mid \; \ord_{\p_1^{n_1} \p_2^{n_2} \ldots \p_k^{n_k}}( \beta ).  \]
  That is, \[  p_j^{m_j} \cdot \ord_{\p_j^{n_j}}(\beta) \; \mid \; p_j^{m_j} \cdot \ord_{\p_1^{n_1} \p_2^{n_2} \ldots \p_k^{n_k}}( \beta ). \]
  Together with (\ref{eq:p-j}), we obtain
  \[ p_j^{\lceil\frac{n_j}{e_j}\rceil} \;\mid\; p_j^{m_j} \cdot \ord_{\p_1^{n_1} \p_2^{n_2} \ldots \p_k^{n_k}}( \beta ) \quad \forall 1 \le j \le k. \]
  Note that all primes in $\mathcal{P}$ are distinct.
  So we have \[ \prod_{p \in \mathcal{P}} p^{\max\big\{ \big\lceil\frac{n_j}{e_j}\big\rceil: j \in \mathcal{I}_p \big\}} \; \mid \; p_1^{m_1} p_2^{m_2} \cdots p_k^{m_k} \cdot \ord_{\p_1^{n_1} \p_2^{n_2} \ldots \p_k^{n_k}}( \beta ). \]
  It follows that
  \[ \ord_{\p_1^{n_1} \p_2^{n_2} \ldots \p_k^{n_k}}( \beta ) \ge \frac{ 1 }{ p_1^{m_1} p_2^{m_2} \cdots p_k^{m_k} } \prod_{p \in \mathcal{P}} p^{\max\big\{ \big\lceil\frac{n_j}{e_j}\big\rceil: j \in \mathcal{I}_p \big\}}, \]
  as desired.
\end{proof}

\section{Proof of Theorem \ref{theorem-complex-plane}}\label{sec:proof}

In this section, we will prove Theorem \ref{theorem-complex-plane}.
Let $K$ be an imaginary quadratic field and let $\mathcal{O}_K$ be the ring of algebraic integers of $K$.

For $\alpha \in \mathcal{O}_K$ with $|\alpha| > 1$, recall that \[ \mathcal{D}_\alpha = \bigcup_{n=1}^\f \frac{\mathcal{O}_K}{\alpha^n}. \]
We write the ideal $\alpha \mathcal{O}_K$ as
\[ \alpha \mathcal{O}_K = \p_1^{b_1} \p_2^{b_2} \cdots \p_{\ell}^{b_\ell} \] a product of powers of distinct prime ideals.
For $z \in \mathcal{D}_\alpha$, there exists $n \in \N$ such that $\alpha^n z \in \calO_K$.
It follows that $z \p_1^{nb_1} \p_2^{nb_2} \cdots \p_{\ell}^{nb_\ell} \subset \calO_K$.
We define $\boldsymbol{n}_z$ to be a tuple $(n_1, n_2, \ldots, n_\ell) \in \N_0^\ell$ with smallest summation $n_1 + n_2 + \cdots + n_\ell$ such that
\[ z \p_1^{n_1} \p_2^{n_2} \cdots \p_{\ell}^{n_\ell} \subset \calO_K. \]
For any tuple $\boldsymbol{n}=(n_1, n_2, \ldots, n_\ell) \in \N_0^\ell$, define \[ \mathcal{D}_\alpha^{\n} := \big\{ z \in \mathcal{D}_\alpha: \n_z = (n_1, n_2, \ldots, n_\ell)\big\}. \]
Then the set $\mathcal{D}_\alpha$ can be written as disjoint unions of $\mathcal{D}_\alpha^{\n}$, i.e.,
\begin{equation}\label{eq:D-alpha-disjoint-union}
  \mathcal{D}_\alpha = \bigcup_{\n \in \N_0^\ell} \mathcal{D}_\alpha^{\n}.
\end{equation}

For $1 \le j \le \ell$, let $p_j \in \N$ be the unique prime contained in $\mathfrak{p}_j$, let $e_j$ be the index of ramification of $\p_j$, and let $f_j$ be the residual degree of $\p_j$.
Note that the degree of $K/\Q$ is $2$.
By \cite[Proposition 4.4]{book-ANT-2022}, we have $e_j \in \{1,2\}$ and $f_j \in \{1, 2\}$ for all $1 \le j \le \ell$.
Let $\mathcal{P}=\big\{ p_j: 1 \le j \le \ell \big\}$ and for $p \in \mathcal{P}$, let $\mathcal{I}_p = \big\{ 1 \le j \le \ell: p_j = p \big\}$.

\begin{proof}[Proof of Theorem \ref{theorem-complex-plane} {\rm(i)}]
  For any tuple $\n=(n_1, n_2, \ldots, n_\ell) \in \N_0^\ell$ and $z \in \mathcal{D}_\alpha^{\n}$, we have $ z \p_1^{n_1} \p_2^{n_2} \cdots \p_{\ell}^{n_\ell} \subset \calO_K$.
  Note that $p_j \in \p_j$ for all $1 \le j \le \ell$. So, $z p_1^{n_1} p_2^{n_2} \cdots p_{\ell}^{n_\ell} \in \calO_K$.
  It follows that \[ \mathcal{D}_\alpha^{\n} \subset \frac{\calO_K}{p_1^{n_1} p_2^{n_2} \cdots p_{\ell}^{n_\ell}}. \]
  This implies that the set $\mathcal{D}_\alpha^{\n}$ is uniformly discrete.
  Note that $S_{\beta,A}$ is a compact set.
  Thus, we conclude that \[ \#\big( \mathcal{D}_\alpha^{\n} \cap S_{\beta,A} \big) < +\f \quad \forall \n \in \N_0^\ell. \]
  In order to show $\#\big( \mathcal{D}_\alpha \cap S_{\beta,A} \big) < +\f$, by (\ref{eq:D-alpha-disjoint-union}) it suffices to find $n_0 \in \N$ such that for any tuple $\n=(n_1, n_2, \ldots, n_\ell) \in \N_0^\ell$ with $n_1 + n_2 + \cdots + n_\ell \ge n_0$, we have \[\mathcal{D}_\alpha^{\n} \cap S_{\beta,A}=\emptyset.\]

  Assume that $s = \dim_{\mathrm{H}} S_{\beta,A} < 1$.
  Choose $\ep>0$ such that $s+\ep < 1$.
  Let $c_1$ and $c_2$ be the constants appearing in Propositions \ref{lem:order-upper-bound} and  \ref{prop:order}, respectively.
  Take a large integer $n_0 \in \N$ such that
  \begin{equation}\label{eq:n-0}
    (\sqrt{2})^{n_0(1-s-\ep)/\ell} > \frac{c_1}{c_2}.
  \end{equation}
  Fix a tuple $\n=(n_1, n_2, \ldots, n_\ell) \in \N_0^\ell$ with $n_1 + n_2 + \cdots + n_\ell \ge n_0$.
  Suppose that $\mathcal{D}_\alpha^{\n} \cap S_{\beta,A}\ne\emptyset$, and we take $z \in \mathcal{D}_\alpha^{\n} \cap S_{\beta,A}$.

  For $p \in \mathcal{P}$, we have \[ p \in \prod_{j \in \mathcal{I}_p} \p_j^{e_j}, \]
  and hence, \[ p^{\max\big\{ \big\lceil\frac{n_j}{e_j}\big\rceil: j \in \mathcal{I}_p \big\}} \in \prod_{j \in \mathcal{I}_p} \p_j^{n_j}. \]
  It follows that \[ \prod_{p \in \mathcal{P}} p^{\max\big\{ \big\lceil\frac{n_j}{e_j}\big\rceil: j \in \mathcal{I}_p \big\}} \in \p_1^{n_1} \p_2^{n_2} \ldots \p_k^{n_k}. \]
  Note that $z \p_1^{n_1} \p_2^{n_2} \cdots \p_{\ell}^{n_\ell} \subset \calO_K$.
  So we have \[ z \prod_{p \in \mathcal{P}} p^{\max\big\{ \big\lceil\frac{n_j}{e_j}\big\rceil: j \in \mathcal{I}_p \big\}} \in \calO_K.\]
This means that we can write $z$ as \[ z = \frac{v}{ \prod_{p \in \mathcal{P}} p^{\max\big\{ \big\lceil\frac{n_j}{e_j}\big\rceil: j \in \mathcal{I}_p \big\}} } \quad\text{with}\; v \in \calO_K. \]
Since $z \in S_{\beta,A}$, by Proposition \ref{lem:order-upper-bound}, the complex number $z$ has an eventually periodic coding $a_1 a_2 \ldots a_k (a_{k+1} a_{k+2} \ldots a_{k+m})^\f\in A^\N$ with
\begin{equation}\label{eq:m-upper}
  m \le c_1 \bigg( \prod_{p \in \mathcal{P}} p^{\max\big\{ \big\lceil\frac{n_j}{e_j}\big\rceil: j \in \mathcal{I}_p \big\}} \bigg)^{s+\ep}.
\end{equation}
Note that $A \subset \calO_K$. We have
\begin{align*}
  z & = \sum_{j=1}^{k} \frac{a_j}{\beta^j} + \Big( 1 + \frac{1}{\beta^m} + \frac{1}{\beta^{2m}} + \cdots \Big) \sum_{j=1}^{m} \frac{a_{k+j}}{\beta^{k+j}} \\
  & = \sum_{j=1}^{k} \frac{a_j}{\beta^j} + \frac{\beta^m}{\beta^m -1} \sum_{j=1}^{m} \frac{a_{k+j}}{\beta^{k+j}} \\
  & \in \frac{\calO_K}{\beta^k(\beta^m -1)}.
\end{align*}
That is, $z \beta^k(\beta^m-1) \in \calO_K$.

We claim that
\begin{equation}\label{eq:claim-m}
  \beta^m -1 \in \p_1^{n_1} \p_2^{n_2} \cdots \p_{\ell}^{n_\ell}.
\end{equation}
Write \[ \beta^k(\beta^m-1) \calO_K = \prod_{j=1}^{\ell} \p_j^{d_j} \prod_{j=\ell+1}^r \p_j^{d_j} \]
as a product of powers of distinct prime ideals, where $d_1, \ldots, d_\ell \in \N_{0}$ and $d_{\ell+1}, \ldots, d_r\in \N$.
Then we have
\[ \big( \p_1^{n_1} \p_2^{n_2} \cdots \p_{\ell}^{n_\ell} \big) + \Big( \beta^k(\beta^m-1) \calO_K \Big) = \prod_{j=1}^\ell \mathfrak{p}_{j}^{\min\{n_j,d_j\}}. \]
It follows that \[ \big( z \p_1^{n_1} \p_2^{n_2} \cdots \p_{\ell}^{n_\ell} \big) + \Big( z \beta^k(\beta^m-1) \calO_K \Big) = z\prod_{j=1}^\ell \mathfrak{p}_{j}^{\min\{n_j,d_j\}}. \]
Note that $z\p_1^{n_1} \p_2^{n_2} \cdots \p_{\ell}^{n_\ell} \subset \mathcal{O}_K$ and $ z\beta^k(\beta^m-1) \in  \calO_K$.
We obtain \[ z\prod_{j=1}^\ell \mathfrak{p}_{j}^{\min\{n_j,d_j\}} \subset \calO_K. \]
By the minimality of $\n_z$, we have $d_j \ge n_j$ for all $1 \le j \le \ell$.
It follows that \[ \p_1^{n_1} \p_2^{n_2} \cdots \p_{\ell}^{n_\ell} \; \mid \; \beta^k(\beta^m-1) \calO_K. \]
Since $\alpha$ and $\beta$ are relatively prime, we conclude that \[ \p_1^{n_1} \p_2^{n_2} \cdots \p_{\ell}^{n_\ell} \; \mid \; (\beta^m-1) \calO_K, \]
which implies the claim (\ref{eq:claim-m}).

By (\ref{eq:claim-m}), we have $\ord_{\p_1^{n_1} \p_2^{n_2} \cdots \p_{\ell}^{n_\ell}}(\beta) \mid m$.
By Proposition \ref{prop:order}, we obtain
\begin{equation}\label{eq:m-lower}
  m \ge \ord_{\p_1^{n_1} \p_2^{n_2} \cdots \p_{\ell}^{n_\ell}}(\beta) \ge c_2 \prod_{p \in \mathcal{P}} p^{\max\big\{ \big\lceil\frac{n_j}{e_j}\big\rceil: j \in \mathcal{I}_p \big\}}.
\end{equation}
Together with (\ref{eq:m-upper}), we obtain \[ \bigg( \prod_{p \in \mathcal{P}} p^{\max\big\{ \big\lceil\frac{n_j}{e_j}\big\rceil: j \in \mathcal{I}_p \big\}} \bigg)^{1-s-\ep} \le \frac{c_1}{c_2}. \]
Note that $p \ge 2$ for all $p \in \mathcal{P}$, and $e_j = 1 \;\text{or}\; 2$ for all $1 \le j \le \ell$.
We have
\begin{align*}
  \prod_{p \in \mathcal{P}} p^{\max\big\{ \big\lceil\frac{n_j}{e_j}\big\rceil: j \in \mathcal{I}_p \big\}} & \ge 2^{\max\big\{ \big\lceil\frac{n_j}{e_j}\big\rceil:\; 1 \le j \le \ell \big\}} \\
  & \ge 2^{\max\big\{ \frac{n_j}{2}: \;1 \le j \le \ell \big\}} \\
  & \ge 2^{\frac{n_1 + n_2 +\cdots+ n_\ell}{2\ell}} \\
  & \ge (\sqrt{2})^{n_0/\ell}.
\end{align*}
Thus, we get \[ (\sqrt{2})^{n_0(1-s-\ep)/\ell} \le \frac{c_1}{c_2}, \]
which contradicts (\ref{eq:n-0}).

Therefore, for any tuple $\n=(n_1, n_2, \ldots, n_\ell) \in \N_0^\ell$ with $n_1 + n_2 + \cdots + n_\ell \ge n_0$, we have $\mathcal{D}_\alpha^{\n} \cap S_{\beta,A}=\emptyset$.
The proof is complete.
\end{proof}

\begin{proof}[Proof of Theorem \ref{theorem-complex-plane} {\rm(ii)}]
  The proof is essentially identical with that of Theorem \ref{theorem-complex-plane} {\rm(i)}.
  We need a better estimation of upper bound in $(\ref{eq:m-upper})$.

  Assume that $\calO_K$ is a unique factorization domain.
  Then every prime ideal is a principal ideal.
  This means that for $1 \le j \le \ell$, we can find $\alpha_j \in \calO_K$ such that $\p_j = \alpha_j \calO_K$.
  Thus, for $\n=(n_1, n_2, \ldots, n_\ell) \in \N_0^\ell$, each element $z \in \mathcal{D}_\alpha^{\n}$ has the form
  \begin{equation}\label{eq:z-new-form}
    z = \frac{v}{\alpha_1^{n_1} \alpha_2^{n_2} \cdots \alpha_\ell^{n_\ell}} \quad\text{for some}\; v \in \calO_K.
  \end{equation}
  We also assume that $\overline{\alpha}$ and $\alpha$ are relatively prime.
  This implies that all primes $p_1, p_2, \ldots, p_\ell$ are distinct. And we have $\p_j \ne \overline{\p_j}$ for all $1 \le j \le \ell$, where $\overline{\p_j}$ is the conjugate prime ideal of $\p_j$.
  By \cite[Theorem 4.11]{book-ANT-2022}, we have $e_j = f_j = 1$ for all $1 \le j \le \ell$.
  It follows that
  \begin{equation}\label{eq:p-j-alpha-j}
    p_j = N(\p_j) = N(\alpha_j \calO_K) = |\alpha_j|^2
  \end{equation}
  The remaining argument is the same with that in the proof of Theorem \ref{theorem-complex-plane} (i).

  Assume that $s = \dim_{\mathrm{H}} S_{\beta,A} < 2$.
  Choose $\ep>0$ such that $s+\ep < 2$.
  Take a large integer $n_0 \in \N$ such that
  \begin{equation}\label{eq:n-0-2}
    (\sqrt{2})^{n_0(2-s-\ep)} > \frac{c_1}{c_2},
  \end{equation}
  where $c_1$ and $c_2$ are the constants appearing in Propositions \ref{lem:order-upper-bound} and \ref{prop:order}, respectively.
  Fix a tuple $\n=(n_1, n_2, \ldots, n_\ell) \in \N_0^\ell$ with $n_1 + n_2 + \cdots + n_\ell \ge n_0$.
  Suppose that $\mathcal{D}_\alpha^{\n} \cap S_{\beta,A}\ne\emptyset$, and we take $z \in \mathcal{D}_\alpha^{\n} \cap S_{\beta,A}$.

  By Proposition \ref{lem:order-upper-bound} and (\ref{eq:z-new-form}), the complex number $z$ has an eventually periodic coding $a_1 a_2 \ldots a_k (a_{k+1} a_{k+2} \ldots a_{k+m})^\f$ with
  \begin{equation}\label{eq:m-upper-2}
    m \le c_1|\alpha_1^{n_1} \alpha_2^{n_2} \cdots \alpha_\ell^{n_\ell}|^{s+\ep} = c_1 (p_1^{n_1} p_2^{n_2} \cdots p_\ell^{n_\ell} \big)^{(s+\ep)/2},
  \end{equation}
  where the last equality follows from (\ref{eq:p-j-alpha-j}).
  As we have done in the proof of Theorem \ref{theorem-complex-plane} (i), we have \[ \beta^m - 1 \in  \p_1^{n_1} \p_2^{n_2} \cdots \p_{\ell}^{n_\ell}. \]
  Note that all primes $p_1, p_2, \ldots, p_\ell$ are distinct, and $e_j = 1$ for all $1 \le j \le \ell$.
  By Proposition \ref{prop:order}, we have
  \[ m \ge c_2 p_1^{n_1} p_2^{n_2} \cdots p_\ell^{n_\ell}. \]
  Together with (\ref{eq:m-upper-2}), we obtain \[ (p_1^{n_1} p_2^{n_2} \cdots p_\ell^{n_\ell}\big)^{(2-s-\ep)/2} \le \frac{c_1}{c_2}. \]
  It follows that \[ (\sqrt{2})^{n_0(2-s-\ep)} \le (p_1^{n_1} p_2^{n_2} \cdots p_\ell^{n_\ell}\big)^{(2-s-\ep)/2} \le \frac{c_1}{c_2},  \]
  which contradicts (\ref{eq:n-0-2}).

  Thus, for any tuple $\n=(n_1, n_2, \ldots, n_\ell) \in \N_0^\ell$ with $n_1 + n_2 + \cdots + n_\ell \ge n_0$, we have $\mathcal{D}_\alpha^{\n} \cap S_{\beta,A}=\emptyset$.
  Therefore, by (\ref{eq:D-alpha-disjoint-union}), we conclude that $\#\big( \mathcal{D}_\alpha \cap S_{\beta,A} \big) < +\f$.
\end{proof}

\section*{Acknowledgements}
W.~Li was supported by NSFC No.~12471085 and Science and Technology Commission of Shanghai Municipality (STCSM) No. 22DZ2229014.
Z.~Wang was supported by NSFC No.~12501110 and the China Postdoctoral Science Foundation No.~2024M763857.

\bibliographystyle{abbrv}
\bibliography{Rational-in-Cantor}

\end{document}